\documentclass[12pt,reqno]{amsart}

\usepackage{amsmath, amssymb, amsthm}

\begin{document}

 \bibliographystyle{plain}
 \newtheorem{theorem}{Theorem}
 \newtheorem{lemma}[theorem]{Lemma}
 \newtheorem{corollary}[theorem]{Corollary}
 \newtheorem{problem}[theorem]{Problem}
 \newtheorem{conjecture}[theorem]{Conjecture}
 \newtheorem{definition}[theorem]{Definition}
 \newtheorem{prop}[theorem]{Proposition}
 \numberwithin{equation}{section}
 \numberwithin{theorem}{section}

 \newcommand{\mo}{~\mathrm{mod}~}
 \newcommand{\mc}{\mathcal}
 \newcommand{\rar}{\rightarrow}
 \newcommand{\Rar}{\Rightarrow}
 \newcommand{\lar}{\leftarrow}
 \newcommand{\lrar}{\leftrightarrow}
 \newcommand{\Lrar}{\Leftrightarrow}
 \newcommand{\zpz}{\mathbb{Z}/p\mathbb{Z}}
 \newcommand{\mbb}{\mathbb}
 \newcommand{\B}{\mc{B}}
 \newcommand{\cc}{\mc{C}}
 \newcommand{\D}{\mc{D}}
 \newcommand{\E}{\mc{E}}
 \newcommand{\F}{\mathbb{F}}
 \newcommand{\G}{\mc{G}}
  \newcommand{\ZG}{\Z (G)}
 \newcommand{\FN}{\F_n}
 \newcommand{\I}{\mc{I}}
 \newcommand{\J}{\mc{J}}
 \newcommand{\M}{\mc{M}}
 \newcommand{\nn}{\mc{N}}
 \newcommand{\qq}{\mc{Q}}
 \newcommand{\PP}{\mc{P}}
 \newcommand{\U}{\mc{U}}
 \newcommand{\X}{\mc{X}}
 \newcommand{\Y}{\mc{Y}}
 \newcommand{\itQ}{\mc{Q}}
 \newcommand{\sgn}{\mathrm{sgn}}
 \newcommand{\C}{\mathbb{C}}
 \newcommand{\R}{\mathbb{R}}
 \newcommand{\T}{\mathbb{T}}
 \newcommand{\N}{\mathbb{N}}
 \newcommand{\Q}{\mathbb{Q}}
 \newcommand{\Z}{\mathbb{Z}}
 \newcommand{\A}{\mathcal{A}}
 \newcommand{\ff}{\mathfrak F}
 \newcommand{\fb}{f_{\beta}}
 \newcommand{\fg}{f_{\gamma}}
 \newcommand{\gb}{g_{\beta}}
 \newcommand{\vphi}{\varphi}
 \newcommand{\whXq}{\widehat{X}_q(0)}
 \newcommand{\Xnn}{g_{n,N}}
 \newcommand{\lf}{\left\lfloor}
 \newcommand{\rf}{\right\rfloor}
 \newcommand{\lQx}{L_Q(x)}
 \newcommand{\lQQ}{\frac{\lQx}{Q}}
 \newcommand{\rQx}{R_Q(x)}
 \newcommand{\rQQ}{\frac{\rQx}{Q}}
 \newcommand{\elQ}{\ell_Q(\alpha )}
 \newcommand{\oa}{\overline{a}}
 \newcommand{\oI}{\overline{I}}
 \newcommand{\dx}{\text{\rm d}x}
 \newcommand{\dy}{\text{\rm d}y}
\newcommand{\cal}[1]{\mathcal{#1}}
\newcommand{\cH}{{\cal H}}
\newcommand{\diam}{\operatorname{diam}}
\newcommand{\bx}{\mathbf{x}}

\parskip=0.5ex

\title[Codimension one cut-and-project nets]{Equivalence classes of codimension one cut-and-project nets}
\author{Alan~Haynes}
\thanks{Research supported by EPSRC grants EP/J00149X/1 and EP/L001462/1.}
\address{Department of Mathematics, University of York, York, UK}
\email{alan.haynes@york.ac.uk}

\allowdisplaybreaks

\begin{abstract}
We prove that in any totally irrational cut-and-project setup with codimension (internal space dimension) one, it is possible to choose sections (windows) in non-trivial ways so that the resulting sets are bounded displacement to lattices. Our proof demonstrates that for any irrational $\alpha$, regardless of Diophantine type, there is a collection of intervals in $\R/\Z$ which is closed under translation, contains intervals of arbitrarily small length, and along which the discrepancy of the sequence $\{n\alpha\}$ is bounded above uniformly by a constant.
\end{abstract}

\maketitle

\section{Introduction}
A {\em separated net} $Y$ in $\R^d$ is a set for which there exist constants $r,R>0$ such that for any distinct points $y,y'\in Y$, the distance from $y$ to $y'$ is at least $r$, and for any $x\in\R^d$, the ball of radius $R$ centered at $x$ contains at least one point of $Y$. Separated nets (also called Delone sets) occur as prominent features in the theories of quasi-periodic functions and mathematical quasicrystals. For surveys of these connections the reader is encouraged to read \cite{BaakMood2004,Meye1995,Mood2008}.

Attempting to understand the deformation properties of separated nets is an attractive venture, with potential mathematical and real world applications (see for example \cite{McMu1998}), which has been undertaken by a number of authors. We say that two separated nets $Y$ and $Y'$ are {\em bounded displacement equivalent} (or simply, BD) if there is a bijection $f:Y\rar Y'$ with the property that
\[\sup_{y\in Y}|f(y)-y|<\infty.\]
Equivalently, $Y$ and $Y'$ are BD to one another if the points of one set can be moved bijectively to the other, moving each point by at most some fixed constant amount. Similarly, we say that two separated nets $Y$ and $Y'$ are {\em bi-Lipschitz equivalent} (or simply, BL) if there are constants $c,C>0$ and a bijection $f:Y\rar Y'$ with the property that, for all pairs of distinct points $y_1,y_2\in Y$,
\[c<\frac{|f(y_1)-f(y_2)|}{|y_1-y_2|}<C.\]

It is elementary to check that BD and BL are equivalence relations on the collection of separated nets, and that two separated nets which are BD equivalent are also BL equivalent. An important and slightly less obvious fact is that any two lattices in $\R^d$ of the same covolume are BD equivalent (see \cite[Proposition 2.1]{HaynKellWeis2013}). It follows that every lattice in $\R^d$ is BL equivalent to $\Z^d$. Gromov asked (see \cite{BuraKlei2002} and \cite{McMu1998} for a history of the problem) whether every separated net in $\R^d$ is BL equivalent to $\Z^d$. This question was answered in 1998, independently by Burago and Kleiner \cite{BuraKlei1998} and McMullen \cite{McMu1998}, who demonstrated that there are separated nets which are not BL to a lattice. Subsequent to this discovery, much attention has been paid to understanding the BD and BL equivalence classes of an important subset of separated nets known as cut-and-project sets, which we now describe.

Suppose that $V$ is a $d-$dimensional subspace of $\R^k$, let $\pi:\R^k\rar\R^k/\Z^k$ be the canonical projection, and suppose that $\mc{S}\subseteq\R^k/\Z^k$ is the image under $\pi$ of an open, bounded subset of a $(k-d)-$dimensional plane in $\R^k$ which is everywhere transverse to $V$. For each $x\in\R^k$ define $Y=Y_{\mc{S},x}\subseteq V$ by
\[Y_{\mc{S},x}=\{v\in V: \pi (v+x)\in \mc{S}\}.\]
If $V$ is {\em totally irrational} (equivalently, if $\pi (V)$ is dense in $\R^k/\Z^k$) then $Y$ will be a separated net (see \cite[Section 2.2]{HaynKellWeis2013}). Technically $Y$ is a subset of $V$, but by making a choice of coordinates, which amounts to a linear transformation that does not change the BD or BL class of $Y$, we can think of $Y$ as a subset of $\R^d$. We refer to $Y_{\mc{S},x}$ as the {\em cut-and-project set} associated to $k,V,\mc{S},$ and $x$, and if it is a separated net then we will call $Y$ a {\em cut-and-project net}. Readers who are familiar with the more traditional definition of cut-and-project sets should not be concerned, as it is not difficult to verify that ours produces the same sets.

It is an open problem (see the introduction of \cite{BuraKlei2002}) to determine whether or not every cut-and-project net is BL to a lattice. Burago and Kleiner \cite{BuraKlei2002} proved that when $d=2$ and $k=3$, if $V$ satisfies a certain Diophantine condition then $Y$ will be BL to a lattice. Solomon \cite{Solo2011} proved that centers of tiles in Penrose tilings, which correspond to cut-and-project nets with $k=5$ and $d=2$, are in fact BD to a lattice.  Most recently, in \cite[Theorem 1.1]{HaynKellWeis2013} it was shown in greater generality that for any dimensions $d$ and $k$, with very mild assumptions on the sets $\mc{S}$, almost every subspace $V$ (in the sense of the natural measure on the Grassmannian manifold) satisfies a Diophantine condition which ensures that the corresponding separated net is BL to a lattice. For any $d$ and $k$ with $d\ge 2$, and with more restrictions on $\mc{S}$, it was proved in \cite[Theorem 1.2]{HaynKellWeis2013} that almost every $V$ satisfies a Diophantine condition ensuring that $Y$ is BD to a lattice. There, it was also proved that for almost every parallelotope $\mc{S}$, there is a non-empty set of subspaces $V$ for which the associated sets $Y$ are not BD to a lattice.

By contrast, in this paper we are going to show that when $d=k-1$, for any $V$ and $x$, there are always choices of $\mc{S}$ called {\em special intervals} for which the associated sets $Y$ are BD to a lattice.
\begin{theorem}\label{thm.BDDEquiv}
If $Y=Y_{\mc{S},x}$ is a cut-and-project set with $d=k-1$, and if $\mc{S}$ is a special interval for $V$, then there is a bijection $f$ from $Y$ to a lattice satisfying
\begin{equation}\label{eqn.BDconstant}
\sup_{y\in Y}|f(y)-y|\le K_V|\mc{S}|^{-1}.
\end{equation}
Here $K_V$ is a constant which may depend on $V$, but does not depend on any other parameters involved.
\end{theorem}
Special intervals will be defined precisely in Section \ref{sec.ReturnTimes}. For now it suffices to know that they are a collection of intervals in $\R/\Z$, which depend on $V$, whose lengths can be taken to be arbitrarily small. As a corollary of Theorem \ref{thm.BDDEquiv} we immediately obtain the following result.
\begin{corollary}\label{cor.BDequiv}
If $Y=Y_{\mc{S},x}$ is a cut-and-project set with $d=k-1$, and if $\mc{S}$ is a finite disjoint union of special intervals, then $Y$ is BD to a lattice.
\end{corollary}
We remark that Corollary \ref{cor.BDequiv} could be deduced from a theorem of Hecke and Ostrowski (see \cite{Heck1922}, \cite{Ostr1927/30}, or the more widely available \cite{Kest1966/67}, for a discussion of their result), after the initial reductions given in the next section (although our proof would in general give much better constants for the BD map). What makes Theorem \ref{thm.BDDEquiv} important is the strength of inequality (\ref{eqn.BDconstant}). This allows us to draw a much stronger conclusion about BL equivalence than previous results would have allowed.
\begin{theorem}\label{thm.BL equiv}
If $Y=Y_{\mc{S},x}$ is a cut-and-project set with $d=k-1$, and if there is a $C>0$ such that $\mc{S}$ is a countable disjoint union of special intervals, with no more than $C$ special intervals of any given length, then $Y$ is BL to a lattice.
\end{theorem}
Note that in all of our results we have dispensed with the Diophantine condition which was present in previous work. As the reader will discover, the key to making this possible is Theorem \ref{thm.ReturnTimes} below. The theorem implies that for any irrational $\alpha,$
\begin{equation}\label{eqn.UnifDiscOverSpecInts}
\sup_{N\in\N}\sup_{\mc{J}\in\R/\Z}\left|\#\{n\in\N : n\le N, n\alpha\in\mc{J}\}-N|\mc{J}|\right|\le K,
\end{equation}
where $K$ is a universal constant, and the inner $\sup$ is taken over all special intervals $\mc{J}$ for $\alpha$.

In order to explain how this fits in the context of previous work, we digress for a moment. It is of long standing importance in many mathematical and scientific disciplines to be able to quantify how evenly distributed a sequence of real numbers is, modulo $1$. One way of doing this is to define, for $N\in\N$, the {\em discrepancy} $D_N$ of a sequence $\{x_n\}_{n=1}^\infty\subseteq\R/\Z$ by
\[D_N(\{x_n\})=\sup_{\mc{I}\subseteq\R/\Z}\left|\#\{1\le n\le N:x_n\in \mc{I}\}-N|\mc{I}|\right|,\]
where the supremum is taken over all intervals $\mc{I}$ in $\R/\Z$. A useful and fairly precise estimate for $D_N$ can then be obtained by using the Erd\H{o}s-Tur\'{a}n Inequality, which states that for any $M\in\N$,
\[D_N\ll \frac{N}{M}+\sum_{m=1}^M\frac{1}{m}\left|\sum_{n=1}^N e(mx_n)\right|,\]
where $e(x)=\exp (2\pi i x)$. Considering the case when $x_n=n\alpha$, the exponential sum here is nicely bounded by
\begin{equation}\label{eqn.ExpSum}
\sum_{n=1}^N e(mn\alpha)\ll\min\left\{N,\frac{1}{\|m\alpha\|}\right\},
\end{equation}
where $\|x\|=\min_{a\in\Z}|x-a|$. This analysis leads to an upper bound for $D_N$ which necessarily depends on how well approximable $\alpha$ is by rational numbers. For example, if $\alpha$ is a Liouville number then there will be a significant number of times when the right hand side of (\ref{eqn.ExpSum}) is $N$, and this will lead to a large estimate of discrepancy. Unfortunately the estimate which is obtained from this argument is not far from the truth. Even in the best possible scenario (when $\alpha$ is {\em badly approximable}), it is known (e.g. see \cite{Schm1972}) that
\[\limsup_{N\rar\infty}\frac{D_N(\{n\alpha\})}{\log N} > 0,\]
and for well approximable numbers the situation can be much worse. This was essentially the source of the limitation of our techniques in \cite{HaynKellWeis2013}, which made it necessary for us to impose a Diophantine condition on $V$.

On the other hand, it was proved by Hecke \cite{Heck1922} and Ostrowksi \cite{Ostr1927/30} that if $\alpha$ is any irrational number and $\mc{I}$ is an interval of length $\|\ell\alpha\|,$ for some $\ell\in\N$, then there is a constant $C(\ell)$ such that
\[\sup_{N\in\N}\left|\#\{n\in\N : n\le N, n\alpha\in\mc{I}\}-N|\mc{I}|\right|\le C(\ell).\]
However their estimates for the constants $C(\ell)$ tend to infinity with $\ell$ (see the discussion at the beginning of \cite[Section 4]{Kest1966/67}). As we will see, our special intervals are a subset of the intervals considered by Hecke and Ostrowski, obtained by restricting $\ell$ to an infinite subsequence of positive integers that depends on the continued fraction expansion to $\alpha$. For this subset we prove the uniform bound recorded in (\ref{eqn.BDconstant}).

\vspace*{.1in}

{\em Acknowledgments:} The author would like to thank Henna Koivusalo for her detailed examination and comments concerning this work, Barak Weiss and Yann Bugeaud for their valuable feedback and advice, and Robert Tichy for pointing out an important reference.

\section{Initial reduction of the problem}\label{sec.InitialReduction}
Throughout the remainder of the paper we assume we are working with a $k-1$ dimensional subspace $V$ of $\R^k$, so that our section $\mathcal{S}$ is a connected segment of a curve in $\R^k$ which is everywhere transverse to $V.$ The $k=2$ cases of what we are going to say are in general much easier, so in much of what follows we will implicitly assume that $k\ge 3$.

If $V$ is not totally irrational then it is contained in a proper rational subspace $W$ of $\R^k$, which contains a lattice that is a subgroup of $\Z^k$. Since we are working in codimension one it follows that $W=V$ and that $Y$ is either empty or is a lattice. Therefore we assume without loss of generality that $V$ is totally irrational.

By deforming $\mathcal{S}$ continuously in the directions parallel to $V$ we may assume that $\mathcal{S}$ is a line segment which is parallel to one of the standard basis vectors for $\R^k$ (there is at least one such vector not lying in the subspace $V$). There is no loss of generality in this assumption for what we are trying to prove, as it causes each of the points in the corresponding set $Y$ to move by at most some fixed finite amount. This deformation will not add any points to $Y$, and it will only delete points if the length of the resulting line segment is greater than $1$. In the latter case we can write the line segment as a disjoint union of line segments, each having length at most $1$, and we can apply our arguments below to show that each of the resulting nets is BD (or BL) to a lattice. Then by appealing to \cite[Proposition 2.4]{HaynKellWeis2013}, we can conclude that the set $Y$ is also BD (or BL, by using the proof of the Proposition mentioned) to a lattice. Therefore we assume after deforming and relabelling, that $\mathcal{S}$ is an interval of length less than $1$, parallel to the $k^{\rm{th}}$ standard basis vector $e_k$.

Finally, we assume that $\mathcal{S}$ is contained in the line in $\R^k$ spanned by $e_k$. Again, there is no loss of generality in this assumption, because our proof in what follows applies to $Y_{\mathcal{S},x}$ for all $x\in\R^k$. Most of our analysis below will take place on integer translates of the line containing $\mathcal{S}$, and it will sometimes be convenient to identify $\mathcal{S}$ with the interval $\mathcal{I}\subseteq\R$ to which it corresponds. {\em Therefore, when we say that $\mc{S}$ is a special interval, or a union of special intervals of a certain form, as we have in the Introduction, this is to be interpreted as a statement about $\mc{I}$.}

Now suppose that $x\in\R^k$ and choose $\alpha_1,\ldots ,\alpha_k\in\R$ so that
\[V+x=\left\{\left(y_1,\ldots ,y_{k-1},\alpha_k+\sum_{i=1}^{k-1}y_i\alpha_{i}\right):y_1,\ldots ,y_{k-1}\in\R\right\}.\]
Then after rotation and re-scaling (which introduces a scaling factor that depends possibly on $V$) we see that $Y_{\mathcal{S},x}$ is BD equivalent to the set
\begin{equation}\label{eqn.ReductionToY'}
Y'=\left\{(n_1,\ldots ,n_{k-1})\in\Z^{k-1}:\alpha_k+\sum_{i=1}^{k-1}n_i\alpha_{i}\in \mathcal{I}~\mathrm{mod}~ 1\right\}.
\end{equation}
Note that our assumption that $V$ is totally irrational implies that at least one of the numbers $\alpha_1,\ldots ,\alpha_{k-1}$ is irrational.

Our plan of proof is to carefully analyze, for any choice of $\gamma\in\R$, and $M,N\in\N$, the number of $N\le n\le N+M$ satisfying
\[n\alpha_1-\gamma\in\mathcal{I}~\mathrm{mod}~1.\]
When $\mathcal{I}$ is what we call a special interval, this will allow us to explicitly define a BD map from $Y'$ to a lattice. For a large class of more general intervals (those mentioned in the hypotheses of Theorem \ref{thm.BL equiv}), our estimates for the quantities above will allow us to accurately count the number of points of $Y'$ in large hypercubes in $\R^{k-1}$. Then we will complete the proof of Theorem \ref{thm.BL equiv} by appealing to a known sufficient condition for BL equivalence, which we now describe.

For any separated net $Y\subseteq\R^d$ and for any $\lambda,\rho>0$, define
\[D_{Y}(\rho, \lambda) = \sup_B \left|\frac{\#(Y\cap B)}{\lambda|B|}-1\right|,\]
where the supremum is taken over all hypercubes $B \subset \R^d$ with side length $\rho$. The result we will use, which is due to Burago and Kleiner for the $d=2$ case and Aliste-Prieto, Coronel, and Gambaudo for the $d>2$ case, is the following.
\begin{theorem}[\cite{BuraKlei2002},\cite{AlisCoroGamb2011}]\label{thm.BuragoKleiner}
If there is a $\lambda>0$ for which
\[\sum_{k=1}^\infty D_Y(2^k, \lambda)<\infty\]
then $Y$ is BL to a lattice.
\end{theorem}
We remark that there is a necessary and sufficient condition due to Laczkovich \cite[Theorem 1.1]{Lacz1992} for determining whether or not a separated net is BD to a lattice. For the case in our problem when $\mathcal{I}$ is a special interval we could have completed the proof of BD equivalence by appealing to this condition, but we chose instead to demonstrate a BD bijection.

\section{Analysis of return times}\label{sec.ReturnTimes}

\subsection{Continued fractions}
We write the simple continued fraction expansion of an irrational real number $\alpha$ as
\begin{align*}
\alpha = a_0 + \cfrac{1}{a_1+
            \cfrac{1}{a_2+
             \cfrac{1}{a_3+\dotsb}}}=[a_0; a_1, a_2, a_3, \dots ],
\end{align*}
where $a_0$ is an integer and $a_1, a_2, \dots $ is a sequence of positive integers uniquely determined by
$\alpha$. The rational numbers
\[\frac{p_k}{q_k}=[a_0;a_1,\ldots ,a_k],\quad k\ge 0,\]
are the principal convergents to $\alpha$, and it is assumed that $p_k$ and $q_k$ are coprime and that $q_k>0$ for all $k$. For $k\ge 0$ we also write
\begin{equation*}
D_k=q_k\alpha-p_k.
\end{equation*}
We have by the basic properties of continued fractions that for $k\ge 1$,
\begin{equation}\label{eqn:cfprop1}
p_{k+1}=a_{k+1}p_k+p_{k-1},\qquad q_{k+1}=a_{k+1}q_k+q_{k-1},\quad\text{ and}
\end{equation}
\begin{equation}\label{eqn:D_kineq}
(-1)^kD_k=\left|q_k\alpha-p_k\right|\le\frac{1}{q_{k+1}}.
\end{equation}
The following lemma describes what we will refer to as the Ostrowski expansion (with respect to $\alpha$) of an integer.
\begin{lemma}\label{lem:intcfexp}\cite[Section II.4]{RockSzus1992}
Suppose $\alpha\in\R$ is irrational. Then for every $n\in\N$ there is a unique integer $M\ge0$ and a unique sequence $\{c_{k+1}\}_{k=0}^\infty$ of integers such that $q_M\le n< q_{M+1}$ and
\begin{equation}\label{Ostexp1}
n=\sum_{k=0}^\infty c_{k+1}q_k,
\end{equation}
\begin{equation*}
\text{with }~0\le c_1<a_1,\quad 0\le c_{k+1}\le a_{k+1}\ \text{ for } \ k\ge 1,
\end{equation*}
\begin{equation*}
c_k=0 \quad \text{whenever}\quad c_{k+1}=a_{k+1}\ \text{ for some }k\ge 1, \qquad\text{and}
\end{equation*}
\begin{equation*}
c_{k+1}=0 \quad \text{for}  \quad k>M.
\end{equation*}
\end{lemma}
For convenience we will consider the integer $0$ to have the Ostrowski expansion given by taking $c_{k+1}=0$ for all $k$. There is a similar expansion for real numbers which uses the $D_k$'s in place of the $q_k$'s. In what follows $\{x\}$ denotes the fractional part of a real number $x$.
\begin{lemma}\label{lem:realcfexp}\cite[Theorem II.6.1]{RockSzus1992}
Suppose $\alpha\in\R\setminus\Q$ has continued fraction expansion as above. For any $\beta\in [-\{\alpha\},1-\{\alpha\})\setminus (\alpha\Z+\Z)$ there is a unique sequence $\{b_{k+1}\}_{k=0}^{\infty}$ of integers such that
\begin{equation}\label{Ostexp2}
\beta=\sum_{k=0}^\infty b_{k+1}D_k,
\end{equation}
\begin{equation*}
\text{with }~0\le b_1<a_1,\quad 0\le b_{k+1}\le a_{k+1}\ \text{ for } \ k\ge 1,\qquad\text{and}
\end{equation*}
\begin{equation*}
b_k=0 \quad \text{whenever}\quad b_{k+1}=a_{k+1}\ \text{ for some }k\ge 1.
\end{equation*}
\end{lemma}
The relevance of these expansions is explained by the following result, which can be deduced from the arguments in \cite[Section II.6]{RockSzus1992} (a rigorous proof can also be found in \cite{BereHaynVaalVela2014}, which should soon be available electronically).
\begin{lemma}\label{lem:inhomappbnd}
Let $\alpha\in\R\setminus\Q$ and suppose that $\gamma\in [-\{\alpha\},1-\{\alpha\})\setminus(\alpha\Z+\Z).$ Choose an integer $n\in\N$ and, referring to the expansions (\ref{Ostexp1}) and (\ref{Ostexp2}), write $\delta_{k+1}=c_{k+1}-b_{k+1}$ for $k\ge 0$. If there is an integer $m\ge 4$ such that $\delta_{k+1}=0$ for all $k<m$, then
\begin{equation*}
\left|n\alpha-\sum_{k=0}^Mc_{k+1}p_k-\gamma\right|\le \frac{3\max(1,|\delta_{m+1}|)}{q_{m+1}}.
\end{equation*}
\end{lemma}
Finally, to simplify some formulas we follow the notation in \cite{RockSzus1992} to define, for fixed irrational $\alpha$ and for $k\ge 0$,
\[\zeta_k=[a_k;a_{k+1},\ldots]\quad\text{and}\quad\xi_k=\frac{q_{k-1}}{q_k}.\]

\subsection{Special intervals and blocks of gaps}
Now we present the definition of special intervals. Throughout our discussion we assume that $\alpha$ is an irrational real number in the unit interval, and we also think of $\alpha$ as an element of $\R/\Z.$ For each integer $m\in\N$ we define $\A (m)\subseteq\Z$ to be the set of non-negative integers $n$ with Ostrowski expansions of the form
\[n=\sum_{k=m}^\infty c_{k+1}q_k.\]
Then, for each $m\in\N$ and for each $\gamma\in\R/\Z$ we define a subset $\J(m,\gamma)$ of $\R/\Z$ by
\[\J(m,\gamma)=\gamma+\overline{\{n\alpha : n\in\A(m)\}}.\]
It is not difficult to see, in light of Lemmas \ref{lem:realcfexp} and \ref{lem:inhomappbnd}, that each set $\J(m,\gamma)$ is an interval in $\R/\Z$, and these intervals are what we refer to as special intervals. In this section we will begin to see why they merit this name.

We wish to study the collection of integers $n$ for which $n\alpha$ lies in a particular special interval. In order to do this first we will investigate the structure of the sets $\A(m)$. We write each set $\A(m)$ as
\[\A(m)=\{n_i(m)\}_{i\ge 0},~\text{with}~n_i(m)<n_{i+1}(m)~\text{for all}~i,\]
and we claim that this also orders the elements of $\A(m)$ lexicographically according to the digits in their Ostrowski expansion. In fact, for any two integers $n,n'\in\N$ with Ostrowski digits $\{c_{k+1}\}$ and $\{c_{k+1}'\}$, respectively, we have that $n<n'$ if and only if there exists an $m\ge 0$ so that $c_{k+1}=c'_{k+1}$ for all $0\le k<m$, and such that $c_{m+1}<c'_{m+1}$. If this is not obvious to the reader, it follows from the observation that the Ostrowski expansion in Lemma \ref{lem:intcfexp} (which is unique) can be obtained by using the greedy algorithm, successively choosing the largest possible value of the largest possible digit at each step.

Suppose that $m$ is chosen, and for each $i\ge 1$ let
\[g_i=n_{i}(m)-n_{i-1}(m).\]
If the Ostrowski expansion of $n_i$ (we will suppress the dependence on $m$) has digits $\{c_{k+1}\}_{k\ge 0}$ then, since the sequence $\{n_i\}$ is ordered lexicographically, the number $g_{i+1}$ will equal $q_m$ if $c_{m+1}<a_{m+1}$ and $q_{m-1}$ if $c_{m+1}=a_{m+1}$. To fully capture the pattern of the sequence $\{g_i\}$ we define a sequence of words $\{B_i\}_{i\ge 0}$ on the letters $1$ and $2$, by stipulating that $B_i$ encodes, in order, the occurrences of $q_m$ and $q_{m-1}$ (represented by the letters $1$ and $2$, respectively) in the sequence $\{g_1,\ldots ,g_{M(i)}\}$, where $M(i)$ is the integer satisfying
\[n_{M(i)}=q_{m+i}.\]
In other words, for $i\ge 0$, the block $B_i$ represents a rule which, when read from left to right, tells us the sequence of increments necessary to step through each element of $\{n_i\}$ in increasing order, starting from $0$ and ending at $q_{m+i}$. To give some examples, we have that $B_0=1$, that
\[B_1=\overbrace{1\cdots 1}^{a_{m+1}}2,\]
and that
\[B_2=\underbrace{(\overbrace{1\cdots 1}^{a_{m+1}}2)\cdots (\overbrace{1\cdots 1}^{a_{m+1}}2)}_{a_{m+2}}1.\]
From the recurrence relation (\ref{eqn:cfprop1}) we see that in order to reach $q_{m+i}$ we must apply the sequence of gaps encoded in $B_{i-1}$, $a_{m+i}$ times, and then we must apply the sequence of gaps encoded in $B_{i-2}$ one time. This means that for $i\ge 2$,
\begin{equation}\label{eqn.BlockRecursion}
B_i=\overbrace{B_{i-1}\cdots B_{i-1}}^{a_{m+i}}B_{i-2}.
\end{equation}
Now for each $i\ge 0$ define $r_i/s_i$ to be the reduced rational given by
\[\frac{r_i}{s_i}=[0;a_{m+1},\ldots ,a_{m+i}],\]
with $s_i>0$ and $r_0/s_0$ taken to be $0/1$. Note that the fraction $r_i/s_i$ is the $i$th principal convergent to $\{\zeta_m\}$. The number of $1$'s which occur in $B_0$ equals $s_0$, and the number of $1$'s which occur in $B_1$ equals $s_1$, so it follows from (\ref{eqn:cfprop1}) and (\ref{eqn.BlockRecursion}) that
\begin{equation}\label{eqn.OneCount}
\#\{1\text{'s in }B_i\}=s_i.
\end{equation}
Similarly the number of $2$'s in $B_0$ equals $r_0$ and the number of $2$'s in $B_1$ equals $r_1$, so it follows from (\ref{eqn:cfprop1}) and (\ref{eqn.BlockRecursion})
\begin{equation}\label{eqn.TwoCount}
\#\{2\text{'s in }B_i\}=r_i.
\end{equation}
We conclude this subsection with a proof of the following theorem.
\begin{theorem}\label{thm.GapCount}
There exists a universal constant $K>0$ such that, for all $m\ge 1$ and for all $M\ge 0$,
\[\left|M-n_M(m)\cdot\left(\frac{1+\{\zeta_m\}}{q_m(1+\{\zeta_m\}\xi_m)}\right)\right|\le K.\]
\end{theorem}
\begin{proof}
We encode the entire sequence $\{g_i\}$ into an infinite word with the letters $1$ and $2$, which we denote as $B_\infty$. In addition to the quantities $\{B_i\}$ and $r_i/s_i$, which were defined for $i\ge 0$, let us define the block $B_{-1}=2$ and the integers $r_{-1}=1$ and $s_{-1}=0$. Then the recursion (\ref{eqn.BlockRecursion}) holds for $i\ge 1$ and the formulas (\ref{eqn.OneCount}) and (\ref{eqn.TwoCount}) hold for $i\ge 0$.

For any $i\ge 0$, we can encode $B_\infty$ as a word in the blocks $B_i$ and $B_{i-1}$, and this encoding will have the property that no two $B_{i-1}$'s ever appear consecutively. This can be verified by considering what we have said in the previous section about Ostrowski expansion. Starting from $n_0$ we can partition the sequence $\A(m)$ into ordered subsets of consecutive elements which end at integers for which the Ostrowski digit $c_{m+i+1}$ has just been incremented. Each of these subsets corresponds to the block $B_i$, and this plan will continue until we reach the end of a subset where the final integer has $c_{m+i+1}=a_{m+i+1}$. Then we choose the next subset in our partition to be one corresponding to the block $B_{i-1}$. The integer at the end of this block will have $c_{m+i+1}=0$, allowing us to begin again with subsets corresponding to $B_i$ blocks.

Now suppose that $M>1$ and let $W_M$ denote the prefix of $B_\infty$ of length $M$. Choose $i_1$ to be the largest integer with the property that $B_{i_1}$ is a prefix of $W_M$. Then, in the encoding of $B_\infty$ with respect to $B_{i_1}$ and $B_{i_1-1}$, the number of complete $B_{i_1}$ blocks which are completely contained in $W_M$ is at most $a_{m+i_1+1}$. This is clear because, in light of (\ref{eqn.BlockRecursion}), if there were more than this number of complete $B_{i_1}$ blocks then there would have been a complete $B_{i_1+1}$ block contained in $W_M$. Now let the number of complete $B_{i_1}$ blocks contained in $W_M$ be $d_{i_1}$, and write
\[W_M=\underbrace{B_{i_1}\cdots B_{i_1}}_{d_{i_1}}W_M'.\]
Based on what we said in the previous paragraph, the word $W_M'$ is either properly contained in a $B_{i_1}$ or $B_{i_1-1}$ block, or it has the form
\begin{equation}\label{eqn.WordDecomp1}
W_M'=B_{i_1-1}W_M'',
\end{equation}
with $W_M''$ properly contained in a $B_{i_1}$ block.

If $W_M'$ is properly contained in a $B_{i_1}$ or $B_{i_1-1}$ block then we may apply the same argument as before, this time choosing $i_2$ to be the largest integer with the property that $B_{i_2}$ is a prefix of $W_M'$, in the encoding of $B_\infty$ with respect to $B_{i_2}$ and $B_{i_2-1}$. It follows that $i_2<i_1$ and as before we will have that the number of complete $B_{i_2}$ blocks contained in $W_M'$, which we call $d_{i_2}$, is at most $a_{m+i_2+1}$.

If $W_M'$ has the form (\ref{eqn.WordDecomp1}) then we take $i_2=i_1-2$ and we apply our argument to $W_M''$. If $W_M''$ has a complete $B_{i_1-1}$ block as a prefix then we let $d_{i_2}$ be one more than the number of $B_{i_2}$ blocks in $W_M''$. If $W_M''$ does not have a complete $B_{i_1-1}$ as a prefix then it must be properly contained in a $B_{i_2}$ block. In this case we let $d_{i_2}=1$ and we choose $i_3$ to be the largest integer with the property that $W_M''$ contains a prefix of $B_{i_3}$.

Continuing in this way we obtain an encoding of $W_M$ of the form
\[W_M=(\underbrace{B_{i_1}\cdots B_{i_1}}_{d_{i_1}})(\underbrace{B_{i_2}\cdots B_{i_2}}_{d_{i_2}})\cdots (\underbrace{B_{i_K}\cdots B_{i_K}}_{d_{i_K}}),\]
with $i_1>\cdots >i_K\ge-1$ and $d_{i_k}\le a_{m+i_k+1}+1$ for all $k$ with $i_k>-1$. If $i_K=-1$ then $d_{i_K}=1$. Counting the number of occurrences of the letters $1$ and $2$ in each of these blocks gives
\begin{equation}\label{eqn.MBlockRep}
M=|W_M|=\sum_{k=1}^Kd_{i_k}(s_{i_k}+r_{i_k}),
\end{equation}
while counting them with the weights $q_m$ and $q_{m-1}$ gives
\begin{equation}\label{eqn.n_MBlockRep}
n_M=\sum_{k=1}^Kd_{i_k}(s_{i_k}q_m+r_{i_k}q_{m-1}).
\end{equation}
Now notice that
\begin{align}\label{eqn.Zeta_mApprox1}
\left|\left(\sum_{k=1}^Kd_{i_k}s_{i_k}\right)\{\zeta_m\}-\left(\sum_{k=1}^Kd_{i_k}r_{i_k}\right)\right|\le \sum_{k=1}^Kd_{i_k}|D_{i_k}(\{\zeta_m\})|
\end{align}
By (\ref{eqn:D_kineq}) the right hand side is bounded above by
\begin{align*}
d_{i_K}|D_{i_K}|+d_{i_{K-1}}|D_{i_{K-1}}|+\sum_{k=1}^{K-2}\frac{a_{m+i_k+1}+1}{s_{i_k+1}}\le 2+\sum_{k=1}^{K-2}\frac{2}{s_{i_k}}.
\end{align*}
Since the quantities $s_i$ must grow at least as fast as the Fibonacci sequence, it is not difficult to show that the right hand side here is bounded above by $10$. Returning to (\ref{eqn.Zeta_mApprox1}) this gives
\begin{equation}\label{eqn.Zeta_mApprox2}
\left|\{\zeta_m\}-\frac{\sum_{k=1}^Kd_{i_k}r_{i_k}}{\sum_{k=1}^Kd_{i_k}s_{i_k}}\right|\le\frac{10}{\sum_{k=1}^Kd_{i_k}s_{i_k}} \ll\frac{1}{M},
\end{equation}
the final inequality coming from (\ref{eqn.MBlockRep}) and the fact that $s_{i_k}>r_{i_k}$, except possibly when $i_k=-1$ (this possibility contributes at most $1$ to the sum, so can be covered by the implied constant).

Now combining (\ref{eqn.MBlockRep}) and (\ref{eqn.Zeta_mApprox2}) gives the formula
\[M=\left(\sum_{k=1}^Kd_{i_k}s_{i_k}\right)\left(1+\{\zeta_m\}+O\left(\frac{1}{M}\right)\right),\]
and rearranging this gives
\[\sum_{k=1}^Kd_{i_k}s_{i_k}=\frac{M}{1+\{\zeta_m\}}+O(1).\]
Returning to (\ref{eqn.n_MBlockRep}) we have that
\begin{align}
n_M&=\left(\sum_{k=1}^Kd_{i_k}s_{i_k}\right)\left(q_m+q_{m-1}\frac{\sum_{k=1}^Kd_{i_k}r_{i_k}}{\sum_{k=1}^Kd_{i_k}s_{i_k}}\right)\nonumber\\
&=\left(\frac{M}{1+\{\zeta_m\}}+O(1)\right)\left(q_m+q_{m-1}\{\zeta_m\}+O\left(\frac{q_{m-1}}{M}\right)\right)\nonumber\\
&=M\left(\frac{q_m(1+\{\zeta_m\}\xi_m)}{1+\{\zeta_m\}}\right)+O(q_m).\label{eqn.n_MFormula}
\end{align}
By rearranging this formula, we arrive at the statement of the theorem. As always in this paper, all implied constants in this proof are universal.
\end{proof}

\subsection{Return times to special intervals}
Our analysis of blocks allows us to prove results about return times of $n\alpha$ to special intervals. First of all, from Theorem \ref{thm.GapCount} we deduce the following result.
\begin{lemma}\label{lem.IntervalLength}
For any $m\in\N$ and $\gamma\in\R/\Z$ we have that
\[|\J(m,\gamma)|=\frac{(1+\{\zeta_m\})}{q_m(1+\{\zeta_m\}\xi_m)}.\]
\end{lemma}
\begin{proof}
We assume without loss of generality that $\gamma=0$, and we claim that
\begin{equation}\label{eqn.ReturnTimesToJ}
\A(m)=\{n\ge 0:n\alpha\in \J (m,0)\}.
\end{equation}
Let us write $\J=\J(m,0)$. It is clear from the definitions that
\[\A(m)\subseteq\{n\ge 0:n\alpha\in \J\},\]
so we only need to prove the reverse inclusion. For each pair of integers $m$ and $c$ satisfying $m\ge 1$, $0\le c< a_1$ when $m=1$, and $0\le c\le a_m$ when $m\ge 2$, we define $A(m,c)$ to be the collection of non-negative integers $n$ with Ostrowski expansions of the form
\[n=cq_{m-1}+\sum_{k=m}^{\infty}c_{k+1}q_k,\]
and we define an interval $J(m,c)$ in $\R/\Z$ by
\[J(m,c)=\overline{\{n\alpha:n\in A(m,c)\}}.\]
Now consider the cases when $m=1$. From (\ref{eqn:D_kineq}) we see that the endpoints of the interval $J(1,0)$ are
\[a_2D_1+a_4D_3+\cdots=(D_2-D_0)+(D_4-D_2)+\cdots=-D_0\]
and
\[a_3D_2+a_5D_4+\cdots=(D_3-D_1)+(D_5-D_3)+\cdots=-D_1.\]
Similarly, for $0< c< a_1$ the endpoints of the interval $J(1,c)$ are
\[cD_0+(a_2-1)D_1+a_4D_3+\cdots=(c-1)D_0-D_1\]
and
\[cD_0+a_3D_2+a_5D_4+\cdots=cD_0-D_1.\]
We have that $-D_0=-\alpha$ and $(a_1-1)D_0-D_1=1-\alpha$, and it follows that these intervals cover $\R/\Z$, with the only overlaps being at their endpoints. The endpoints are not positive integer multiples of $\alpha$, and since $\J\subseteq J(1,0)$, this shows that any integer $n$ for which $n\alpha\in\J$ must have the first digit in its Ostrowski representation equal to $0$.

If $m=1$ then we are done, otherwise we move on to consider the intervals $J(2,c)$. By the same arguments as before, the intervals $J(2,c)$ cover $J(1,0)$, only overlapping at their endpoints (which are not positive integer multiples of $\alpha$). When $m\ge 2$ we have that $\J\subseteq J(2,0)$, which shows that any integer $n$ for which $n\alpha\in\J$ must the first two digits in its Ostrowski expansion equal to zero. Continuing in this way verifies (\ref{eqn.ReturnTimesToJ}).

The rest of the proof follows immediately from Theorem \ref{thm.GapCount}, together with the well known fact that for any irrational $\alpha$, the sequence $\{n\alpha\}_{n\in\N}$ is uniformly distributed modulo $1$.
\end{proof}
Now we are positioned to prove the following result which, as we mentioned in the introduction, is the key to what follows.
\begin{theorem}\label{thm.ReturnTimes}
There is a universal constant $K$ with the property that, for any $m\ge 1, \gamma\in\R/\Z$, and $N\in\N$,
\[\left|\#\{n\in\N : n\le N, n\alpha\in\J(m,\gamma)\}-N|\J(m,\gamma)|\right|\le K.\]
\end{theorem}
\begin{proof}
First we consider the case when $\gamma=-\ell\alpha$ for some $\ell\ge 0$. In this case we have, for any $n\in\N$, that $n\alpha\in\J(m,\gamma)$ if and only if $(n+\ell)\alpha\in\J(m,0)$. By (\ref{eqn.ReturnTimesToJ}) this condition is equivalent to $(n+\ell)\in\A(m)$, which shows that
\begin{equation}\label{eqn.NumberOfReturns}
\#\{1\le n\le N: n\alpha\in\J(m,\gamma)\}=\#\{i\ge 0:\ell\le n_{i}(m)\le N+\ell\}.
\end{equation}
Let $n_I(m)$ be the smallest element of $\A(m)$ satisfying $\ell\le n_I(m)$. Now we proceed as in the proof of Theorem \ref{thm.GapCount}, using much of the notation there and again suppressing the dependence on $m$.

For $M>1$ let $W_M=W_M(I)$ denote the subword of $B_\infty$ which begins at the $I^{\rm{th}}$ letter and extends to the $(I+M)^{\rm{th}}$ letter. Let $i_1$ be the largest integer for which $W_M$ contains a complete $B_{i_1}$ block, in the encoding of $B_\infty$ with respect to $B_{i_1}$ and $B_{i_1-1}$. Let $d_{i_1}$ be the number of $B_{i_1}$ blocks in this encoding which are completely contained in $W_M$ so that, as before, $d_{i_1}\le a_{m+i_1+1}$. Write
\[W_M=W_{M-}'\underbrace{B_{i_1}\cdots B_{i_1}}_{d_{i_1}}W_{M+}',\]
so that neither $W_{M^-}'$ nor $W_{M^+}'$ contains a $B_{i_1}$ block.

The block $W_{M^+}'$ may or may not have $B_{i_1-1}$ as a prefix, but in either case, applying the method in proof of Theorem \ref{thm.GapCount}, we know that we can encode it as a union of blocks of types $i_{2^+}>\cdots >i_{K^+}\ge-1$ using $d_{k^+}$ blocks of each type, with $d_{k^+}\le a_{m+(i_{k^+})+1}+1$ for all $k$ with $i_{k^+}>-1$, and with $d_{i_{K^+}}=1$ if $i_{K^+}=-1$.

The same type of argument applies to $W_{M^-}'$, if we encode it working from right to left. First of all, suppose that $W_{M^-}'$ has $B_{i_1-1}$ as a suffix, but that it is not a subword of $B_{i_1}$. Then it has the form
\[W_{M^-}'=W_{M^-}''B_{i_1-1},\]
with $W_{M^-}''$ properly contained in a $B_{i_1}$ block. In this case we take $i_{2^-}=i_1-1$ and, letting $d_{2^-}$ be the number of complete $B_{i_1-1}$ blocks contained in $W_{M^-}'$, we would have that $d_{2^-}\le a_{m+(i_{2^-})+1}$. This would lead to a decomposition of $W_{M^-}''$ of the form
\[W_{M^-}''=W_{M^-}'''\underbrace{B_{i_{2^-}}\cdots B_{i_{2^-}}}_{(d_{2^-})-1}B_{(i_{2^-})-1},\]
with $W_{M^-}'''$ contained as a proper sub-suffix of a $B_{i_{2^-}}$ block. Returning to the discussion of $W_{M^-}',$ the possibilities that we have not considered are covered in the cases when either $W_{M^-}'$ is already a complete block (in which case we are finished) or when it is not a complete block, but is a subword of a $B_{i_1}$ or $B_{i_1-1}$ block. In the latter case we apply the arguments above to choose $i_{2^-}<i_1$ and $d_{2^-}\le a_{m+(i_{2^-})+1}$, giving us an encoding of the form
\[W_{M^-}''=W_{M^-}'''\underbrace{B_{i_{2^-}}\cdots B_{i_{2^-}}}_{(d_{2^-})}B_{(i_{2^-})-1}.\]
Therefore we are guaranteed after our choice of $i_{2^-}$ to reach one of the two types of encodings above. It follows by induction that we can encode $W'_{M^-}$ as a union of blocks of types $i_{2^-}>\cdots >i_{K^-}\ge-1$ using $d_{k^-}$ blocks of each type, with $d_{k^-}\le a_{m+(i_{k^-})+1}$ for all $k$ with $i_{k^-}>-1$, and with $d_{i_{K^-}}=1$ if $i_{K^-}=-1$.

From the previous three paragraphs, we conclude that we can write $W_M$ as a union of blocks of types $i_{1}>\cdots >i_{K}\ge-1$ using $d_k$ blocks of each type, with $d_k\le 2a_{m+i_k+1}+1$ for all $k$ with $i_k>-1$, and with $d_{i_K}\le 2$ if $i_K=-1$. By exactly the same arguments as in Theorem \ref{thm.GapCount}, and by Lemma \ref{lem.IntervalLength}, we have that
\begin{align}
n_{I+M}-n_I&=M\left(\frac{q_m(1+\{\zeta_m\}\xi_m)}{1+\{\zeta_m\}}\right)+O(q_m)\nonumber\\
&=M|\J(m,\gamma)|^{-1}+O\left(|\J(m,\gamma)|^{-1}\right),\label{eqn.DifferenceOfReturnTimes}
\end{align}
where the implied constant does not depend on any of the parameters involved.

Now if the quantity in (\ref{eqn.NumberOfReturns}) is smaller than $3$ then let $M=2$. Otherwise let $M$ be the largest integer such that $n_{I+M}(m)\le N+\ell$. Then since
\[N=(N+\ell)-\ell=n_{I+M}-n_I+O(q_m),\]
we have that
\begin{align*}
N=M|\J(m,\gamma)|^{-1}+O\left(|\J(m,\gamma)|^{-1}\right).
\end{align*}
Rearranging this equation, we have proved that there is a universal constant $K$ such that
\[\left|M-N|\J(m,\gamma)|\right|\le K.\]
The difference between $M$ and the quantity in (\ref{eqn.NumberOfReturns}) is at most $2$, so this finishes the proof of this theorem, in the case when $\gamma=-\ell\alpha$ for some $\ell\ge 0$.

Since $\alpha$ is irrational the set $\{-\ell\alpha\}_{\ell\ge 0}$ is dense in $\R/\Z$. Therefore, given any $\gamma\in\R/\Z$ and $N\in\N$, we can choose $\ell\in\N$ so that
\[\#\{1\le n\le N : n\alpha\in\J(m,\gamma)\}=\#\{1\le n\le N : n\alpha\in\J(m,-\ell\alpha)\}.\]
All of our implied constants are uniform, so the proof of the theorem in its entirety follows immediately from what we have already shown.
\end{proof}

\section{Proofs of Theorems \ref{thm.BDDEquiv} and \ref{thm.BL equiv}}
We come to the proofs of our main theorems. The proof of Theorem \ref{thm.ReturnTimes} should be indicative of how the proof of our first theorem will proceed.
\begin{proof}[Proof of Theorem \ref{thm.BDDEquiv}]
Recall that, by our arguments in Section \ref{sec.InitialReduction}, it is enough to show that the set $Y'$ from (\ref{eqn.ReductionToY'}) is $BD$ to a lattice. First suppose that $\mc{S}$ is a single special interval, say $\J (m,\gamma)$, and suppose without loss of generality, by relabelling if necessary, that $\alpha_1$ is irrational (see the comment following (\ref{eqn.ReductionToY'})). For each $(k-2)-$tuple of integers $(n_2,\ldots ,n_{k-1})$ write
\[\left\{n_1\in\Z:\alpha_k+\sum_{i=1}^{k-1}n_i\alpha_{i}\in \mathcal{\J}~\mathrm{mod}~ 1\right\}=\{\ell_i(n_2,\ldots ,n_{k-1})\}_{i\in\Z},\]
with $\ell_i<\ell_{i+1}$ and $\ell_{-1}<0\le \ell_0$. Consider the map from $Y'$ to the lattice \[(|\J(m,0)|^{-1})\Z\times\Z^{k-2}\]
defined by
\[(\ell_i(n_2,\ldots ,n_{k-1}),n_2,\ldots ,n_{k-1})\mapsto \left(i\cdot |\J(m,0)|^{-1},n_2,\ldots ,n_{k-1}\right).\]
By (\ref{eqn.DifferenceOfReturnTimes}), together with the comments at the end of the proof of the previous theorem, this map is a BD bijection which moves each point by at most $K|\mc{J}|^{-1}$, for some absolute universal constant $K$. Finally, recalling that the argument in Section \ref{sec.InitialReduction} introduced a scaling factor which depended only on $V$, we conclude that the composition $f$ of the rotation and scaling maps from $Y$ to $Y'$, together with the BD map which we have just constructed, satisfies $(\ref{eqn.BDconstant}).$
\end{proof}
The result in Corollary \ref{cor.BDequiv} now follows from  \cite[Proposition 2.4]{HaynKellWeis2013}. The proof there also explains how, together with the map above, to construct an explicit BD map from $Y$ to a lattice. Finally, we conclude with the proof of our BL result.
\begin{proof}[Proof of Theorem \ref{thm.BL equiv}]
Suppose that $\mc{S}$ satisfies the hypotheses of the theorem, and write $\mc{S}$ as a disjoint union,
\[\mc{S}=\bigcup_{k=1}^\infty J_k,\]
where each $J_k$ is a union of $C_k\le C$ disjoint special intervals of the form $\J(m_k,\cdot),$ with $m_1<m_2<\cdots$. Suppose that $B\subseteq\R^{d}$ is a hypercube of side length $2^K$, for some $K\in\N$, and choose $\ell$ to be the smallest integer with the property that $q_{m_\ell-1}>2^K$. Then we have that $\ell\ll K$ and, furthermore, there is a constant $C'$, depending possibly on $C$, with the property that
\begin{equation}\label{eqn.TailOfUnion}
\bigcup_{k=\ell+1}^\infty J_k
\end{equation}
is contained in a union of at most $C'$ special intervals of the form $\J(m_\ell,\cdot)$. Both of these assertions follow from the fact that the integers $q_k$ grow at least geometrically, which together with Lemma \ref{lem.IntervalLength} puts an upper bound on the length of (\ref{eqn.TailOfUnion}) which is directly proportional to $C$ and $|\J(m_\ell,0)|$. Then, since the smallest gap in a $\J(m_\ell,\cdot)$ interval is $q_{m_\ell-1}$, in any collection of $2^K$ consecutive integers there are at most $C'$ integers $n$ for which $n\alpha\in\mc{S}$.

Using the notation introduced in the previous proof, the number of points in $Y'\cap B$ is equal to
\begin{align*}
\sum_{(n_2,\ldots n_{k-1})\in\Z^{k-2}}\#\{i\in\Z:(\ell_i(n_2,\ldots ,n_{k-1}),n_2,\ldots ,n_{k-1})\in Y'\}.
\end{align*}
For any fixed $(k-2)-$tuple of integers, if the summand above is non-zero then it is equal to
\begin{align*}
\sum_{k=1}^\infty\#\{i\in\Z:\ell_i(n_2,\ldots ,n_{k-1})\in J_k+\gamma\}
\end{align*}
for some $\gamma$ depending on all of the parameters involved. However, regardless of what $\gamma$ is, we can apply Theorem \ref{thm.ReturnTimes} to conclude that this is equal to
\begin{align*}
\sum_{k=1}^\ell(2^K|J_k|+O(1))+O(1)&=2^K\sum_{k=1}^\ell|J_k|+O(K)\\
&=2^K|\mc{S}|+O(K),
\end{align*}
where in the last line we are using the estimate
\[2^K\sum_{k=\ell+1}^\infty|J_k|\ll 2^K|J_\ell|\ll 1.\]
Altogether this gives that
\[\#(Y'\cap B)=2^{(k-1)K}|\mc{S}|+O(K2^{(k-2)K}).\]
Applying Theorem \ref{thm.BuragoKleiner} with $\lambda=|\mc{S}|$ then finishes the proof.
\end{proof}

\end{document}